\documentclass{article}
\usepackage{authblk}
\usepackage{bbm}
\usepackage[colorlinks=true,linkcolor=black,citecolor=black]{hyperref}
\usepackage[numbers]{natbib}
\usepackage{url}
\usepackage{amsmath}
\usepackage{amssymb}
\usepackage{theoremref}
\usepackage{tikz,pgfplots}
\usepackage{mathrsfs}
\usepackage{comment}
\usepackage{multicol}
\usepackage{enumerate} 
\usepackage{graphicx}
\usepackage{tikz}
\usepackage{xcolor}

\usepackage{epstopdf}

\usepackage{caption}
\usepackage{subcaption}

\usepackage[latin1]{inputenc}

\usepackage{amsmath,amsfonts,amsthm,bm}
\usepackage{paralist}
\allowdisplaybreaks

\usepackage{enumitem}
\usepackage[normalem]{ulem} 

\usepackage[a4paper]{geometry}

\newcommand{\Z}{\mathbb{Z}} 
\newcommand{\R}{\mathbb{R}} 
\newcommand{\N}{\mathbb{N}} 

\newcommand{\setu}{\mathfrak{u}}

\newcommand{\rd}{\,\mathrm{d}} 
\newcommand{\sinc}{\mathrm{sinc}} 

\newcommand{\bszero}{\boldsymbol{0}} 
\newcommand{\bsh}{\boldsymbol{h}}    
\newcommand{\bsk}{\boldsymbol{k}}    
\newcommand{\bsu}{\boldsymbol{u}}    
\newcommand{\bsx}{\boldsymbol{x}}    
\newcommand{\bsn}{{\boldsymbol{n}}}    

\newcommand{\bsxi}{\boldsymbol{\xi}}    

\newcommand{\e}{\mathrm{e}}

\newcommand{\W}{\mathrm{W}}          

\DeclareSymbolFont{bbold}{U}{bbold}{m}{n}
\DeclareSymbolFontAlphabet{\mathbbold}{bbold}

\newcommand{\mi}{\mathrm{i}}

\def\citep#1#2{\cite[{#1}]{#2}}

\theoremstyle{plain}
  \newtheorem{theorem}{Theorem}
  
  \newtheorem{lemma}{Lemma}
  
\theoremstyle{definition}
  
  \newtheorem{example}{Example}
\theoremstyle{remark}
  \newtheorem{remark}{Remark}

\newcommand{\RefEq}[1]{~\textup{(\ref{#1})}}

\newcommand{\RefSec}[1]{Section~\textup{\ref{#1}}}

\newcommand{\RefThm}[1]{Theorem~\textup{\ref{#1}}}

\newcommand{\RefFig}[1]{Figure~\textup{\ref{#1}}}

\newcommand{\RefRem}[1]{Remark~\textup{\ref{#1}}}

\definecolor{darkred}{RGB}{139,0,0}
\definecolor{darkgreen}{RGB}{0,100,0}
\definecolor{darkmagenta}{RGB}{180,0,180}
\definecolor{darkblue}{RGB}{0,0,190}

\begin{document}
\bibliographystyle{plain}
\title{Multivariate integration over $\R^s$ with \\ exponential rate of convergence}

\author[1]{Dong T.P. Nguyen}
\author[1]{Dirk Nuyens}

\affil[1]{Department of Computer Science, KU Leuven, Celestijnenlaan 200A, B-3001 Leuven, Belgium} 
\affil[ ]{dong.nguyen@cs.kuleuven.be, dirk.nuyens@cs.kuleuven.be}

\date{\today}
\maketitle

\begin{abstract}

  In this paper we analyze the approximation of multivariate integrals over the Euclidean space for functions which are analytic. We show explicit upper bounds which attain the exponential rate of convergence. We use an infinite grid with different mesh sizes in each direction to sample the function, and then truncate it. In our analysis, the mesh sizes and the truncated domain are chosen by optimally balancing the truncation error and the discretization error.
  
  This paper derives results in comparable function space settings, extended to $\R^s$, as which were recently obtained in the unit cube by Dick, Larcher, Pillichshammer and Wo{\'z}niakowski (2011), see \cite{Dic11}.
  They showed that both lattice rules and regular grids, with different mesh sizes in each direction, attain exponential rates, hence motivating us to analyze only cubature formula based on regular meshes.
  We further also amend the analysis of older publications, e.g., Sloan and Osborn (1987) \cite{Slo87} and Sugihara (1987) \cite{Sug87}, using lattice rules on $\R^s$ by taking the truncation error into account and extending them to take the anisotropy of the function space into account.
  
\end{abstract}

\section{Introduction}\label{Introduction}
We study approximation of the multivariate integral
\begin{align}\label{eq:I}
 I(f)
 &:= 
 \int_{\R^s} f(\bsx) \rd {\bsx}
\end{align}
for functions defined over the $s$-dimensional Euclidean space $\R^s$, which belong to some normed weighted function space, by using a cubature rule using function values of $f$, cf.~\eqref{eq:QhD}.
The function space setting will be formally introduced in \RefSec{sec:WFS}, but let us make some remarks already here.
We stress that we do not integrate explicitly against any specific probability density function, and as such this implicitly requires the function to decay fast enough towards infinity to be integrable, i.e., at least with a polynomial rate strictly larger than $1$, which will be denoted by assumption~\ref{F-poly} in \RefSec{sec:WFS}.
In particular we are interested in approximating the above integral with an exponential rate $\mathcal{O}(\exp(- C_1(s) N^{C_2(s)}))$ or $\mathcal{O}(\exp (-C_3(s) N^{C_4(s)} (\ln N) ^{C_5(s)}))$ in the number of integration nodes $N$, where $C_1(s)$, $C_2(s)$, $C_3(s)$, $C_4(s)$ and $C_5(s)$ are constants depending on $s$.

The Fourier transform of a function $f$ will be denoted by 
\begin{align*}
 \hat f({\bsxi}) 
 &:= 
 \int_{\R^s} f(\bsx) \, \e^{-2 \pi\mi \bsxi \cdot  \bsx} \rd\bsx
 ,
 \qquad
 \bsxi \in \R^s.
\end{align*}
It is well known that the decay of the Fourier transform can be interpreted as a kind of smoothness of the function.
In the one dimensional case the Fourier transform of an analytic function decays faster than any polynomial.
Moreover, according to the Paley--Weiner theorem, if a real-valued analytic function has some continuation into the complex plane and satisfies some decay conditions then its Fourier transform decays exponentially fast, see \cite[Chapter 4, Theorem 2.1, 3.1]{Ste03}.
As a natural extension to the multivariate case, we will therefore further restrict to the case where the Fourier transform decays at least exponential, which will be denoted by assumption~\ref{W-exp} in \RefSec{sec:WFS}.

Our analysis is done for a cubature algorithm which can be described by two consecutive steps.
First we approximate $I(f)$ by
\begin{align}\label{eq:Ih}
 I_{\bsh} (f) 
 := 
 h_1 \cdots h_s \sum_{\bsk \in \Z^s} f(k_1 h_1, \dots, k_s h_s)
 ,
\end{align}
with different mesh sizes $h_j$, $j=1,\ldots,s$, for different dimensions, based on the anisotropy of the Fourier transform of $f$,
and then we truncate the infinite sum to obtain our final quadrature approximation
\begin{align}\label{eq:QhD}
 Q_{\bsh}^{\mathscr{D}_\bsn}(f) 
 := h_1 \cdots h_s \sum_{\bsk \in \mathscr{D}_{\bsn}} f(k_1 h_1, \dots, k_s h_s)
 ,
\end{align}
for some appropriate truncation set $\mathscr{D}_\bsn$ depending on the anisotropic decay of the function itself.
Although this is a straightforward algorithm, it will be essential to determine the mesh sizes $h_j$, $j=1,\ldots,s$, and the truncation set $\mathscr{D}_\bsn$ in an optimal way to balance their error contributions.
We remark that for integration over $\R^s$ there is only little difference between the ``rectangle'', ``trapezoidal'' and ``midpoint'' rule.
Furthermore, as the distinction does not play any role in the further analysis we just denote our rule as the ``trapezoidal rule'', similar to what was argued in \cite{Tre14} in the univariate setting.
We choose the truncation set $\mathscr{D}_\bsn$ such that $-n_j/2 \le k_j \le n_j/2$ for all $j=1,\ldots,s$, and for well chosen truncation points $n_j/2$.
This approach leads to two errors which need to be controlled, more specifically, a discretization error and a truncation error:
\begin{align}\label{eq:total-error}
  |I(f) - Q_{\bsh}^{\mathscr{D}_{\bsn}}(f)| 
  &\leq 
  |I(f) - I_{\bsh}(f)| 
  + 
  |I_{\bsh}(f) - Q_{\bsh}^{\mathscr{D}_{\bsn}}(f)|
  .
\end{align}
The discretization error will be studied in \RefSec{sec:discretization-error} and the truncation error will be studied in \RefSec{sec:truncation-error}.
In order to achieve an exponential rate of convergence we will further assume that the final integrand function decays at least exponential, cf.\ assumption~\ref{F-exp} in \RefSec{sec:WFS}, to control the truncation error and we analyze this in \RefSec{sec:exp}.
This may seem like an unreasonably strong assumption, but
for functions that do not decay this fast we employ a variable transform.
E.g., the so-called exponential and double exponential transforms by Takahasi and Mori \cite{Tak74}.
This leads us to also study functions which decay double exponentially, cf.\ assumption~\ref{F-expexp} in the next section and we analyze its truncation error in \RefSec{sec:expexp}.
The idea from \cite{Tak74} is to use a suitable change of variables to transform an integral over a domain $\Omega \subseteq \R^s$ into another integral over the Euclidean space $\R^s$, for which the integrand decays double exponentially fast towards infinity.
The transformed integral can then be approximated by the trapezoidal rule.

The extreme accuracy of the trapezoidal rule for the integration of one dimensional analytic periodic functions on intervals and on the real line was studied by many researchers, see \cite{Tre14} for a recent overview by Trefethen and Weideman. This is easy to show on periodic intervals. As a hand-waving argument the analysis can be carried over to integration on $\R$ if the function goes to zero fast and smooth enough, as then the truncated integrand is ``nearly periodic'' on the truncated interval.
Also in the one-dimensional case this can be quantified in exact terms by an analysis of the discretization and truncation errors.

We now give a quick overview of existing results.
In the setting of multivariate integration over the unit cube $[0,1]^s$, in a weighted periodic function space with exponentially decaying Fourier series expansions, it was recently shown that a regular grid (with different mesh sizes for the different dimensions), as well as ``good'' rank-$1$ lattice rules, achieve exponential rates of convergence, see Dick et al \cite{Dic11}, and Kritzer et al \cite{Kri14}.
We note explicitly that, with respect to the rate, it is shown in \cite{Dic11} that there is no difference if one takes a grid or a lattice rule in the case of periodic functions over $[0,1]^s$.
The weighted function space over $\R^s$ in the current paper, see \RefSec{sec:WFS}, is modeled after the weighted function space over $[0,1]^s$ in these papers.

Multivariate integration over $\R^s$ with exponential convergence rates was also studied in \cite{Slo87,Sug87} and recently in \cite{Irr15}. In \cite{Slo87, Sug87}, the integral $I(f)$ is approximated by an equal-weight rule based on an infinite lattice. In \cite{Slo87}, Sloan and Osborn also considered a regular grid as a special case of a lattice, a so-called ``cubic lattice'', next to more general lattices and in particular the body centered cubic lattice.
They studied the case where the function is isotropic, and thus only have one mesh size for all dimensions, and do not discuss how to balance discretization and truncation error.
Instead they fix their truncation domain in advance (to be all lattice points inside a sphere around the origin with a fixed radius) and then study the convergence in terms of decreasing the mesh size.
Also in their result it makes no difference, in terms of the exponential rate, if one takes a more general lattice as a point set or a regular grid.
They illustrate their technique numerically on different examples for fixed truncation radius.

In \cite{Irr15}, Irrgeher et al use a weighted tensor product of Gauss--Hermite quadrature rule for integrating a class of analytic functions. In their context, the analyticity of a function is characterized by the exponential decay of its Hermite coefficients. They show an explicit error bound which attains the exponential rate of convergence $\mathcal{O}(\exp(- C_1(s) N^{C_2(s)}))$. Moreover, they study so-called exponential convergence and even uniform exponential convergence with tractability, i.e., study the dependence on the dimension of $C_1(s)$, $C_2(s)$ and the constant in the $\mathcal{O}$ notation. Finally, they show an important result, i.e., necessary conditions on the decay of Hermite coefficients such that these constants become independent of the dimension.

We write $\N := \{1,2,\ldots\}$ and $\N_0 := \{0,1,2,\ldots\}$.
We use the shorthand notation $\{1\mathop{:}s\} := \{1, \ldots, s\}$.
Further we write $\R_+ := \{ x \in \R : x > 0 \}$.

\section{Weighted function spaces}\label{sec:WFS}

First, we introduce a function space in the univariate case.
Let $\nu :\R\to \R_+$ and $\omega :\R \to \R_+$.
We then define a one-dimensional space of integrable functions $E(\nu, \omega)$ 
as follows
\begin{align*}
  E(\nu,\omega)
  &:=
  \left\{ 
    f : \R \rightarrow \R 
    :
    \|f\| 
    := 
    \|f \, \nu\|_{L^\infty(\R)}
    + 
    \|\hat f \, \omega \|_{L^\infty(\R)} 
    <
    \infty
  \right\}
  .
\end{align*}
The idea is that $\nu$ controls the decay of the function and 
$\omega$ controls the decay of the Fourier transform of the function.
The faster $\nu$ and $\omega$ increase for increasing $x$ and $\xi$, the faster the function and its Fourier transform decay.

For the multivariate case, we consider the tensor product. 
Let $\nu : \R^s \to \R_+$ and $\omega : \R^s \to \R_+$.
Then we define
\begin{align*}
 E_s(\nu, \omega) 
 := 
 \left\{
   f: \R^s \rightarrow \R 
   :
   \|f\| 
   := 
   \|f \, \nu\|_{L^\infty(\R^s)}
   + 
   \|\hat f \, \omega\|_{L^\infty(\R^s)}
   < 
   \infty
 \right\}
 .
\end{align*}

In this paper, we will consider two main kinds of decay of the Fourier transform:
\begin{enumerate}[label=\textbf{\textup{(W\arabic*)}},leftmargin=3em]
  \item\label{W-poly} decay with polynomial order, for any $\alpha > 0$: 
   $$
     \omega(\bsxi)
     =
     \prod_{j=1}^s \left(1+|\xi_j|^{1+\alpha}\right),
   $$
   \item\label{W-exp} exponentially decay, with all $a_j,b_j > 0$:
   $$
     \omega(\bsxi)
     =
     \exp\left(\sum_{j=1}^s a_j \, |\xi_j|^{b_j}\right).
   $$
\end{enumerate}
Likewise we will assume three main kinds of decay of the function:
\begin{enumerate}[label=\textbf{\textup{(F\arabic*)}},leftmargin=3em]
 \item\label{F-poly} polynomial decay, for any $\alpha > 0$:
   $$
     \nu(\bsx)
     =
     \prod_{j=1}^s \left(1+|x_j|^{1+\alpha}\right),
   $$
 \item\label{F-exp} exponential decay, with all $c_j,d_j > 0$:
   $$
     \nu(\bsx)
     =
     \exp\left(\sum_{j=1}^s c_j \, |x_j|^{d_j}\right),
   $$
 \item\label{F-expexp} double exponentially decay, with all $c_j,d_j, e_j > 0$: 
   $$
     \nu(\bsx)
     =
     \exp\left(\sum_{j=1}^s e_j \exp\left(c_j \, |x_j|^{d_j}\right)\right).
   $$
\end{enumerate}
The double exponential decay \ref{F-expexp} of the function was introduced  by Takahasi and Mori  \cite{Tak74}. They studied one dimensional integration with singularities at the endpoints. By a suitable change of variable
they transformed the integration into integration over the real line where the transformed integrand decays double exponentially fast and they then applied a trapezoidal rule to approximate the integration over the real line.

The above univariate function space $E(\nu,\omega)$ where $\omega$ satisfies assumption~\ref{W-exp} and assumption~\ref{F-exp} or~\ref{F-expexp} with $d_1=1$ is studied in \cite{Sug97}. It was shown that the trapezoidal rule is nearly optimal, in the sense that the upper bound and the lower bound of the error are nearly equal where the difference depends on the number of function evaluations.

We remark that the combination of assumptions~\ref{W-poly} and~\ref{F-poly} allows the use of the Poisson summation formula \cite{Ste03} which links the discretization of the function to a discretization of its Fourier transform:
\begin{align}\label{eq:Poisson-summation}
  h_1 \cdots h_s \sum_{\bsk \in \Z^s} f(k_1 h_1, \ldots, k_s h_s)
  =
  \sum_{\bsk \in \Z^s} \hat{f}(k_1/h_1, \ldots, k_s/h_s)
  .
\end{align}

\section{The discretization error}\label{sec:discretization-error}

We consider the case when the Fourier transform of the function decays exponentially fast as in~\ref{W-exp}.
To be able to use the Poisson summation formula we also need the function itself to decay at least polynomially as in~\ref{F-poly}.
We note that this condition is also needed for integrability.
The following lemma will be useful in the estimation of both discretization error and truncation error.

\begin{lemma} We present three different estimates.
\begin{enumerate}[font=\upshape, label=(\roman*)]
\item   For $b > 0$ and $h\in (0,1]$, we have
  \begin{align}\label{lem:bound}
    \sum_{k=1}^{\infty} \exp(-k^b/h)
    &\le
     \exp(-1/h) \frac{\e \, \Gamma(1/b)}{b}
    ,
  \end{align}
  where $\e:=\exp(1)$ and $\Gamma(\cdot)$ is the gamma function.

\item   For $c >0$, $d\geq 1$ and $n \geq 0$, we have
\begin{align}\label{lem:sumexpo}
\sum_{k=n}^{\infty} \exp(- c \, k^d) \leq \exp(-c n^d) \left( 1 + \frac{\Gamma(1/d)}{c^{1/d} d}\right)
.
\end{align} 

\item For $\alpha>0$, $c>0$, $d\geq 1$ and $n \geq 0$, we have
\begin{align}\label{lem:sumdoubexpo}
\sum_{k=n}^{\infty} \exp(- \alpha\exp(c \, k^d)) \leq 
\e^{-\alpha}  \, \exp(- \alpha\exp(c \, n^d))
 \, \left( 1 + \frac{\Gamma(1/d)}{(\alpha \, c)^{1/d} d} \right)  
.
\end{align} 

\end{enumerate}

\end{lemma}
\begin{proof} The proofs are as follows.
\begin{enumerate}[label=(\roman*)]
\item
We have the estimate
\begin{align*}
  \sum_{k=1}^{\infty} \exp(-k^b/h) 
  &= \exp(-1/h) \sum_{k=1}^{\infty} \exp\left(-(k^b-1) /h\right) 
  \\
  &\leq
  \exp(-1/h)
  \sum_{k=1}^{\infty} \exp\left(-k^b + 1 \right)
  .
\end{align*}
where the last inequality yields from the condition $h \leq 1$ and $k^b-1 \ge 0$.
Moreover, interpreting the sum as an infinite right-rectangle rule applied to a decreasing function, we have 
\begin{align*} 
  \sum_{k=1}^{\infty} \e^{-k^{b}}    
  \le 
  \int_0^{\infty} \e^{-x^{b}} \rd x
  =
  \frac{1}{b}
  \int_0^{\infty} t^{\frac{1}{b} -1} \e^{-t} \rd t
  =
  \frac{\Gamma(1/b)}{b}
  .
\end{align*}
Combining these gives the desired result~\eqref{lem:bound}.
\item 
Inequality~\eqref{lem:sumexpo} is obtained as follows
\begin{align*}
\sum_{k=n}^{\infty} \exp(- c \, k^d) 
& = \exp(-c n^d) \sum_{k=n}^{\infty} \exp(- c \, (k^d - n^d))
\\
& \leq 
\exp(-c n^d) \sum_{k=n}^{\infty} \exp(- c \, (k-n)^d)
\\
&= 
\exp(-c n^d) \sum_{k=0}^{\infty} \exp(- c \, k^d)
\\
&=
\exp(-c n^d) \left( 1 + \sum_{k=1}^{\infty} \exp(- c \, k^d) \right)
\\
&\leq
\exp(-c n^d) \left( 1 + \int_{0}^{\infty} \exp\left(-c x^{d} \right) \rd x \right)
\\
&=
\exp(-c n^d) \left( 1 + \frac{\Gamma(1/d)}{c^{1/d} d}\right)
.
\end{align*}
The first above inequality is due to the following inequality
\begin{align}\label{lem:ineq2}
x^d -n^d \geq (x-n)^d
,
\end{align}
for any $d \geq 1$ and $x \geq n \geq 0$. 

Indeed, we consider the following function
\begin{align*}
f(x) = x^d -n^d - (x-n)^d
.
\end{align*}
Since
\begin{align*}
\frac{\rd f}{\rd x} = d\, x^{d-1} - d \, (x-n)^{d-1}
\geq 0
,
\end{align*}
for every $x \geq n \geq 0$ and since $f(n) =0$, inequality \eqref{lem:ineq2} is easily obtained from the monotonicity of $f$.

\item We have
\begin{align}\label{lem:ineq1}
\sum_{k=n}^{\infty} \exp(- \alpha\exp(c \, k^d)) 
&= 
 \exp(- \alpha\exp(c \, n^d)) \sum_{k=n}^{\infty} \exp(- \alpha (\exp(c \, k^d) - \exp(c \, n^d)))
 \notag
 \\
 &\leq
  \exp(- \alpha\exp(c \, n^d)) \sum_{k=n}^{\infty} \exp(- \alpha \exp(c \, (k-n)^d))
  \notag
  \\
  &=
   \exp(- \alpha\exp(c \, n^d)) \sum_{k=0}^{\infty} \exp(- \alpha \exp(c \, k^d))
   ,
\end{align}
where the first inequality is due to the following inequality
\begin{align*}
\e^{x^d} - \e^{n^d} - \e^{(x-n)^d} \geq 0
,
\end{align*}
for $d \geq 1$ and $x \geq n \geq 0$. This inequality is easily obtained from the monotonicity of the following function
\begin{align*}
f(x) = \e^{x^d} - \e^{n^d} - \e^{(x-n)^d}
.
\end{align*}
Indeed, for $d \geq 1$ and $x \geq n \geq 0$ 
\begin{align*}
\frac{ \rd f} {\rd x} = d \, x^{d-1} \e^{x^d} - d \, (x-n)^{d-1} \e^{(x-n)^d} \geq d \, x^{d-1} \e^{x^d} - d \, (x-n)^{d-1} \e^{x^d} \geq 0
.
\end{align*}
Moreover, we have
\begin{align}\label{sumdoubexpo}
\sum_{k=0}^{\infty} \exp(- \alpha \exp(c \, k^d))
 &\leq 
\e^{-\alpha} + \int_{0}^{\infty} \exp\left(-\alpha \exp\left(c x^{d} \right)\right) \, \rd x
\notag
 \\
 & \leq
 \e^{-\alpha} +  \int_{0}^{\infty} \exp\left(-\alpha (c\,x^d+1) \right) \, \rd x
 \notag
 \\
 &=
  \e^{-\alpha} \left( 1 + \frac{\Gamma(1/d)}{(\alpha \, c)^{1/d} d} \right)  
,
\end{align} 
where we use $\e^{x}\geq x+1$, for $x\geq 0$ in the second inequality .

Applying the above estimation into \eqref{lem:ineq1}, inequality~\eqref{lem:sumdoubexpo} follows.
\end{enumerate}
This completes the proof.
\end{proof}

We are now ready to prove a bound on the discretization error when we choose $h_j = (a_j h)^{1/b_j}$.

\begin{theorem}\label{thm:discretization}
Suppose $f$ belongs to $E_s(\nu,\omega)$ for some $\nu$ and $\omega$ with conditions \ref{F-poly} and \ref{W-exp} such that for some fixed
$a_j, b_j > 0$, we can take
\begin{align*}
  \omega(\bsxi) 
  =
  \exp\left({\sum_{j=1}^s a_j \, |\xi_j|^{b_j}}\right)
  .
\end{align*}
If we set $h_j = (a_j \, h)^{1/b_j}$, then, for any $h \in (0, 1/ \ln(2\, \e)]$, the following inequality is satisfied 
\begin{align*}
  |I_{\bsh}(f) - I(f)| 
  \leq  
  \, C(s,\omega) \,\|\hat f \, \omega\|_{L^\infty(\R^s)} \, \exp(-1/h)
  ,
\end{align*}
where 
\begin{align*}
  C(s, \omega)
  =
  C(s, b_1, \ldots, b_s)
  &:=
  2\,\e \sum_{\emptyset \ne \setu \subseteq \{1:s\}} \prod_{j\in\setu} \frac{\Gamma(1/b_j)}{b_j}
  .
 \end{align*}
\end{theorem}
\begin{proof}
 We first note the obvious identity $I(f) = \hat f(\bszero)$.
 By the assumptions on the decay of the function values and the decay of the Fourier transform we can use the Poisson summation formula\RefEq{eq:Poisson-summation} to rewrite\RefEq{eq:Ih} and obtain
 \begin{align*}
   I_{\bsh}(f) 
   = 
   \sum_{\bsk \in \Z^s} \hat f(k_1/h_1, \dots, k_s/h_s).
  \end{align*}
Thus, we have
\begin{align*}
 |I_{\bsh}(f) - I(f)| 
 &=
 \left| \sum_{\bsk \in \Z^s \setminus \{\bszero\}} \hat f(k_1/h_1, \dots, k_s/h_s) \, \frac{\omega(k_1/h_1,\ldots,k_s/h_s)}{\omega(k_1/h_1,\ldots,k_s/h_s)}\right|
 \notag
 \\ 
 &\leq
 \|\hat f \, \omega \|_{L^\infty(\R^s)} \,
 \sum_{\bsk \in \Z^s \setminus \{\bszero\}} \exp\left({-\sum_{j=1}^{s} a_j \, |k_j|^{b_j} h_j^{-b_j}}\right)
  .
\end{align*}
To evenly distribute the leading terms in the sum we fix $h = h_j^{b_j}/a_j$ for each $j$, then,
we can bound the sum using~\eqref{lem:bound} as
\begin{align*}
 \sum_{\bsk \in \Z^s \setminus \{\bszero\}} \exp\left({-\sum_{j=1}^{s} |k_j|^{b_j} /h }\right)
 &=
 -1 +
   \prod_{j=1}^s \left( 1 + 2 \sum_{k=1}^{\infty} \exp(-k^{b_j}/h) \right)
   \notag
 \\
 &\leq
 -1 +
   \prod_{j=1}^s \left( 1 + \exp(-1/h) \frac{2\, \e}{b_j} \Gamma\left(\frac{1}{b_j}\right) \right)
   \notag
 \\
 &=
  \sum_{\emptyset \ne \setu \subseteq \{1:s\}}  (2\,  \e \, \exp(-1/h) )^{|\setu|} \prod_{j\in\setu} \frac{\Gamma(1/b_j)}{b_j}
  \notag
  \\
  & =
  2 \, \e \, \exp(-1/h) \sum_{\emptyset \ne \setu \subseteq \{1:s\}}  (2\,  \e \, \exp(-1/h) )^{|\setu| - 1} \prod_{j\in\setu} \frac{\Gamma(1/b_j)}{b_j}
  \notag
  \\
  &\leq 
   2 \, \e \, \exp(-1/h) \sum_{\emptyset \ne \setu \subseteq \{1:s\}} \prod_{j\in\setu} \frac{\Gamma(1/b_j)}{b_j}
   ,
\end{align*}
where for the last inequality we used $h \leq 1/ \ln(2\, \e)$.
\end{proof}

\section{The truncation error}\label{sec:truncation-error}

We consider the following truncation domain
\begin{align*}
 \mathscr{D}_{\bsn} 
 = 
 \left\{ \bsk \in \Z^s : -n_j/2 \leq k_j \leq n_j/2 \text{ for all } j=1, \ldots, s\right\},
\end{align*}
where $n_j \in \N_0$, $j=1,\dots, s$, and $n_j/2$ are called the truncation points.
Without loss of generality we consider the case where $n_1, \dots, n_s$ are even integers.
The total number of function evaluations is then given by
\begin{align*}
n=(n_1+1) \cdots (n_s+1).
\end{align*}

We could choose the truncation points differently for the left hand side and the right hand side. 
However, we assume it is easy to use a centering mapping, e.g., as in \cite{Sin13}, 
to transform the function into a new function with symmetric decay around the origin. From the bound\RefEq{eq:total-error} it is clear that if the function decays not at least exponentially at infinity then our proposed algorithm is not efficient. For example, for the one dimensional function $f(x) = (1+x^2)^{-1}$, it will take thousands of function evaluations  to achieve a few digits of accuracy, see \cite{Tre14}. 
In order to achieve an exponential convergence rate it is required that the function decays at least exponentially fast as in assumption~\ref{F-exp}.
We consider this case in \RefSec{sec:exp}.

For functions that are smooth enough in the sense of assumption~\ref{W-exp} but unfortunately decay slower, i.e., decay at most of polynomial order~\ref{F-poly}, we propose to use a variable transformation such that the function decays at least exponential~\ref{F-exp}, or even double exponential~\ref{F-expexp}.
This will be studied in \RefSec{sec:expexp}.

\subsection{Functions that decay exponentially fast}\label{sec:exp}

The truncation error for the function that decays exponentially fast will be estimated by the following theorem.

\begin{theorem} \label{thm:ECtruncation}
Suppose $f$ belongs to $E_s(\nu,\omega)$ for $\nu$ with condition \ref{F-exp} such that for some fixed $c_j> 0$ and $d_j \geq 1$, we can take
\begin{align*}
  \nu(\bsx) 
  =
  \exp\left({\sum_{j=1}^s c_j \, |x_j|^{d_j}}\right)
  .
\end{align*}
Then, for any $h_j >0$ and $n_j \geq 0$, for $j=1,\dots,s$, the following inequality is satisfied 
\begin{align}\label{expo:ineq8}
 |I_{\bsh} (f) - Q_{\bsh}^{\mathscr{D}_{\bsn}}(f)|
 & \leq
 C(s, \nu, \bsh) \, \|f \, \nu\|_{L^\infty(\R^s)}
\exp\left(- \inf_j c_j h_j^{d_j} \left(\frac{n_j+1}{2} \right)^{d_j}\right)
 ,
\end{align}
where 
\begin{align}\label{Csnuh}
C(s, \nu, \bsh) 
&:=
2\,s\,
  \prod_{j=1}^{s}
  \left(
  h_j+ 2\,\frac{\Gamma(1/d_j)}{d_j c_j^{1/d_j}}
  \right)
   .
\end{align}
\end{theorem}
\begin{proof}
From~\eqref{eq:Ih} and~\eqref{eq:QhD} and the definition of the norm using $\nu$, we have
\begin{align}\notag\label{expo:ineq5}
 |I_{\bsh} (f) - Q_{\bsh}^{\mathscr{D}_{\bsn}}(f)| 
 &= 
 \left| h_1 \cdots h_s \sum_{\bsk \in \Z^s \setminus\mathscr{D}_{\bsn}} f(k_1 h_1, \dots, k_s h_s)\right|
 \\\notag
 &\le
 h_1 \cdots h_s \sum_{\bsk \in \Z^s \setminus\mathscr{D}_{\bsn}} \left| f(k_1 h_1, \dots, k_s h_s) \frac {\nu(k_1h_1, \dots, k_sh_s)}{\nu(k_1h_1, \dots, k_sh_s)} \right|
 \\
 &\leq
 \|f \, \nu\|_{L^\infty(\R^s)} \, h_1 \cdots h_s  \sum_{\bsk \in \Z^s \setminus\mathscr{D}_{\bsn}} \exp\left(-\sum_{j=1}^s c_j |k_j\,h_j|^{d_j}\right)
 .
\end{align}
Now we evaluate the sum in the last inequality. We have
\begin{align}\label{expo:ineq4}
& \sum_{\bsk \in \Z^s \setminus\mathscr{D}_{\bsn}} \exp\left(-\sum_{j=1}^s c_j |k_j\,h_j|^{d_j}\right)
& \leq 
2 \sum_{i=1}^s
 \left(
 \sum_{k_1 \in \Z} \cdots \sum_{k_i= \frac{n_i}{2}+1} ^{\infty} \cdots\sum_{k_s \in \Z}
\exp\left(-\sum_{j=1}^s c_j |k_j\,h_j|^{d_j} \right)
\right)
.
\end{align}
For each of the $i=1,\ldots,s$, since $d_i \geq 1$, then using~\eqref{lem:sumexpo}, we have
\begin{align}\label{expo:ineq6}
\sum_{k_i=\frac{n_i}{2}+1}^{\infty} \exp\left(-c_i k_i^{d_i}\,h_i^{d_i} \right)
&\leq 
\exp\left(- c_i h_i^{d_i} \left(\frac{n_i}{2} +1 \right)^{d_i}\right) 
\left(1+ \frac{\Gamma(1/d_i)}{d_i c_i^{1/d_i} h_i} \right)
\notag
\\
&\leq 
\exp\left(- c_i h_i^{d_i} \left(\frac{n_i+1}{2} \right)^{d_i}\right) 
\left(1+2\, \frac{\Gamma(1/d_i)}{d_i c_i^{1/d_i} h_i} \right)
.
\end{align}

Moreover, we have
\begin{align}\label{expo:ineq7}
\sum_{k_i \in \Z}
\exp\left(-c_i h_i ^{d_i} |k_i|^{d_i} \right) 
&=
 1 + 2 \sum_{k_i=1}^{\infty} \exp\left(-c_i h_i ^{d_i} k_i^{d_i} \right)
 \notag 
 \\
&\leq 
 1 + 2 \int_{0}^{\infty} \exp\left(-c_i h_i ^{d_i} x^{d_i} \right) \, \rd x
 \notag
 \\
 & = 
 1+ 2\,\frac{\Gamma(1/d_i)}{d_i c_i^{1/d_i} h_i}
 .
\end{align}

Applying \eqref{expo:ineq6} and \eqref{expo:ineq7} into \eqref{expo:ineq4}, we get
\begin{align}\label{expo:ineq10}
& \sum_{\bsk \in \Z^s \setminus\mathscr{D}_{\bsn}}  \exp\left(-\sum_{j=1}^s c_j |k_j\,h_j|^{d_j}\right)
\notag
\\
&\qquad\qquad\leq 
2 \sum_{i=1}^s
 \exp\left(- c_i h_i^{d_i} \left(\frac{n_i+1}{2} \right)^{d_i}\right) 
  \prod_{j=1}^{s}
  \left(
	1+ 2\,\frac{\Gamma(1/d_j)}{d_j c_j^{1/d_j} h_j}
  \right)
\notag
\\
&\qquad\qquad\leq 
2\,s\, \exp\left(- \inf_{i} c_i h_i^{d_i} \left(\frac{n_i+1}{2} \right)^{d_i}\right)
  \prod_{j=1}^{s}
  \left(
  1+ 2\,\frac{\Gamma(1/d_j)}{d_j c_j^{1/d_j} h_j}
  \right)
.
\end{align}
Applying \eqref{expo:ineq10} into \eqref{expo:ineq5}, the claim follows.
\end{proof}
Now we prove the first main result. Let us denote
\begin{align}\label {C*}
&B(s) := \sum_{j=1}^{s} \frac{1}{b_j}
, 
\notag
\\
&D(s) := \sum_{j=1}^{s} \frac{1}{d_j}
, 
\notag
 \\
&  C_* := \inf_{j} C_j := \inf_{j} \frac{c_j^{1/d_j} a_j^{1/b_j}}{2}
,
\end{align}
and
\begin{align}\label{C}
C_\sharp:=  \inf_j  \left(\frac{C_*}{2}\right) ^{d_j}
.
\end{align}
\begin{theorem} \label{TheoEC}
Suppose $f$ belongs to $E_s(\nu,\omega)$ for some $\nu$ and $\omega$ with conditions \ref{F-exp} and \ref{W-exp} such that for some fixed $a_j, b_j, c_j> 0$ and $d_j \geq 1$, we can take
\begin{align*}
  \nu(\bsx) 
  =
  \exp\left({\sum_{j=1}^s c_j \, |x_j|^{d_j}}\right)
  ,
\end{align*}
and
\begin{align*}
  \omega(\bsxi) 
  =
  \exp\left({\sum_{j=1}^s a_j \, |\xi_j|^{b_j}}\right)
  .
\end{align*}
We set $h_j = \left( a_j h \right)^{1/b_j}$ with
\begin{align*}
 h   
 =   
  N^{-\frac{1}{B(s)+D(s)}}
  C_\sharp^{\frac{-D(s)}{B(s)+D(s)}}
  ,
\end{align*}
and
\begin{align}\label{nj}
n_j+1 
=  
\max \left\{
\left\lfloor  
\frac{C_*} {C_j} \, 
C_\sharp^ {\frac{1}{B(s)+D(s)} \left(\frac{D(s)}{b_j} - \frac{B(s)}{d_j}\right)}
 N^{\frac{1}{B(s)+D(s)} \left(\frac{1}{b_j}+\frac{1}{d_j}\right)} 
\right\rfloor,
1 \right\}
.
\end{align}
Then, for all $N$ sufficiently large, the following inequality is satisfied 
\begin{align*}
|I(f) - Q_{\bsh}^{\mathscr{D}_{\bsn}}(f)| 
\leq
 C(s) 
 \, \exp\left(
 -N^{\frac{1}{B(s)+D(s)}}
 C_\sharp^{\frac{D(s)}{B(s)+D(s)}}
 \right) 
 \, \|f\|,
\end{align*}
where 
\begin{align}\label{Cs}
C(s)
& := 
4\,s\,
  \prod_{j=1}^{s}
  \left(
  a_j^{1/b_j} + 2\,\frac{\Gamma(1/d_j)}{d_j c_j^{1/d_j}}
  \right)
  + 
   4\,\e \sum_{\emptyset \ne \setu \subseteq \{1:s\}} \prod_{j\in\setu} \frac{\Gamma(1/b_j)}{b_j}
,
\end{align}
and $N$ is an upper bound for the total number of function evaluations as
\begin{align*}
(n_1+1)\cdots (n_s+1) \leq N.
\end{align*}
\end{theorem}
\begin{remark}
 $N$ is chosen sufficiently large such that $h \leq 1/ \ln(2\, \e)$, which is the condition on $h$ in \RefThm{thm:discretization}.
\end{remark}
\begin{proof}
We have
\begin{align*}
|I(f) - Q_{\bsh}^{\mathscr{D}_{\bsn}}(f)| 
\leq 
|I(f) - I_{\bsh} (f)| 
+ 
|I_{\bsh} (f) - Q_{\bsh}^{\mathscr{D}_{\bsn}}(f)|
.
\end{align*}
We first study the discretization error $|I(f) - I_{\bsh} (f)|$. 
We see that all the conditions of \RefThm{thm:discretization} are satisfied.
We set $h_j = (a_j h)^{1/b_j}$, $j=1, \dots, s$, such that the discretization error can be estimated as
\begin{align}\label{expo:diserror}
 |I(f) - I_{\bsh} (f)|
 \leq  
 C(s,\omega)\, \exp\left(-1/h\right)\,\|\hat f \, \omega\|_{L^\infty(\R^s)}
 .
\end{align}

Now we estimate the truncation error $|I_{\bsh} (f) - Q_{\bsh}^{\mathscr{D}_{\bsn}}(f)|$. All the conditions of  \RefThm{thm:ECtruncation} are satisfied, then substituting the chosen step size into \eqref{expo:ineq8} the truncation error can now be written as
\begin{align}\label{eq:trunerr}
 |I_{\bsh} (f) - Q_{\bsh}^{\mathscr{D}_{\bsn}}(f)|
 & \leq
 C(s,\nu)\,
 \, \exp\left( - \inf_j c_j \left( a_j h \right)^{\frac{d_j}{b_j}} \left(\frac{n_j+1}{2}\right)^{d_j}\right)\,
 \|f \, \nu\|_{L^\infty(\R^s)}
 ,
\end{align}
where 
\begin{align*}
C(s,\nu) 
=
  2\,s\,
  \prod_{j=1}^{s}
  \left(
  a_j ^{1/b_j} + 2\,\frac{\Gamma(1/d_j)}{d_j c_j^{1/d_j}}
  \right)
  .
\end{align*}

For all $j=1, \dots, s$ define $n_j$ as
\begin{align}\label{eq:nj}
 n_j+1 
 = \max \left\{
 \left\lfloor 
 \frac
 	{ C_* h^\frac{B(s)}{D(s)\, d_j} N^{\frac{1}{D(s)\,d_j}}}
 	{C_j    h^{1/b_j}}
 \right\rfloor,
  1 \right\}
 .
\end{align}
Note that $\lfloor x \rfloor \geq \frac{x}{2}$, for all $x \geq 1$ then 
\begin{align*}
 n_j+1 
 \geq  
 \frac{ C_* h^\frac{B(s)}{D(s)\, d_j} N^{\frac{1}{D(s)\,d_j}}}
 {2\,C_j\,h^{1/b_j}}
 .
\end{align*}
Applying the above inequality into \eqref{eq:trunerr} leads to 
\begin{align}\notag\label{eq:trunerr2}
  |I_{\bsh} (f) - Q_{\bsh}^{\mathscr{D}_{\bsn}}(f)|
  & \leq
   C(s,\nu) \, \exp\left(-\inf_{j}\left(\frac{C_*}{2}\right) ^{d_j}  \left(h^{B(s)}N\right)^{\frac{1}{D(s)}}\right)
   \|f \, \nu\|_{L^\infty(\R^s)}
  \\
  &=
  C(s,\nu)  \, \exp\left(-C_\sharp\, h^{\frac{B(s)}{D(s)}}N^{\frac{1}{D(s)}}  \right)
   \|f\, \nu\|_{L^\infty(\R^s)}
  .
\end{align}

Adding the discretization error~\eqref{expo:diserror} and the truncation error~\eqref{eq:trunerr2} we receive the total error
\begin{align}\label{totalerror}
 |I(f) - I_{\bsh} (f)|
 \leq  
 \widetilde C (s)
 \left[
 \exp\left(-1/h\right) 
 +
 \exp\left(-C_\sharp\, h^{\frac{B(s)}{D(s)}}N^{\frac{1}{D(s)}}  \right)
 \right]
 \|f\|
 ,
\end{align} 
where $\widetilde C(s)$ is given by
\begin{align*}
\widetilde C(s)
=
2\,s\,
  \prod_{j=1}^{s}
  \left(
  a_j^{1/b_j} + 2\,\frac{\Gamma(1/d_j)}{d_j c_j^{1/d_j}}
  \right)
  + 
   2\,\e \sum_{\emptyset \ne \setu \subseteq \{1:s\}} \prod_{j\in\setu} \frac{\Gamma(1/b_j)}{b_j}
   .
\end{align*}
It is clear that when $h$ decreases, the first term in the square bracket of~\eqref{totalerror} increases while the second term decreases. Hence, the optimal $h$ is chosen by balancing these two terms, which leads to
\begin{align*}
h^{-1} 
=  
C_\sharp \, h^{\frac{B(s)}{D(s)}}N^{\frac{1}{D(s)}}
,
\end{align*}
this implies 
\begin{align*}
 h   
 =   
  N^{\frac{-1}{B(s)+D(s)}}
  C_\sharp^{\frac{-D(s)}{B(s)+D(s)}}
  .
\end{align*}
This is equivalent to balancing the orders of magnitude of the discretization error~\eqref{expo:diserror} and the truncation error~\eqref{eq:trunerr2}.

Substituting the chosen $h$ into~\eqref{eq:nj} we get the optimal way of choosing $n_1, \dots, n_s$ as in~\eqref{nj}.
Thus, from~\eqref{totalerror} the total error is given by
\begin{align*}
|I(f) - Q_{\bsh}^{\mathscr{D}_{\bsn}}(f)| 
\leq
 C(s)
 \, \exp\left(
 -N^{\frac{1}{B(s)+D(s)}}
 C_\sharp^{\frac{D(s)}{B(s)+D(s)}}
 \right) 
 \, \|f\|,
\end{align*}
for any $N \in \N$, where $C(s)=2\,\widetilde C (s)$.

It is easy to see that $(n_1+1)\cdots (n_s+1) \leq N$ since
\begin{align*}
 (n_1+1)\cdots (n_s+1) 
 &\leq 
 \prod_{j=1}^{s}
 \frac
 { C_*\,h^\frac{B(s)}{D(s)\, d_j} N^{\frac{1}{D(s)\,d_j}}}
 {C_j\,h^{1/b_j}}
 \leq  
 N.
  \qedhere
\end{align*}

\end{proof}
\begin{remark}\label{reLambda}
In the proof we used $\lfloor x \rfloor \geq \frac{x}{2}$, for $x$ large this is a severe underestimate for $\lfloor x \rfloor$ since for larger $x$ we have $\lfloor x \rfloor \gg x/2$. There exists a better estimate by considering $\lfloor x \rfloor \geq \lambda \, x$, for some $ 0< \lambda \leq 1$. The following part of the proof is easily changed with respect to this estimate. Furthermore, it is easy to see that choosing a sufficient larger $\lambda$ we achieve a better constant in the convergence. In the numerical section we do not use the rough bound but choose a better $\lambda$ (depending on the case).
\end{remark}

\subsection{Functions that decay double exponentially fast}\label{sec:expexp}

Let us denote 
\begin{align*}
e_* = \inf_{j} e_j
.
\end{align*}
The truncation error is bounded by the following theorem.
\begin{theorem} \label{thm:DECtruncation}
Suppose $f$ belongs to $E_s(\nu,\omega)$ for $\nu$  with condition \ref{F-exp} such that for some fixed $c_j> 0$ and $d_j \geq 1$, we can take
\begin{align*}
  \nu(\bsx)
     =
     \exp\left(\sum_{j=1}^s e_j \exp\left(c_j \, |x_j|^{d_j}\right)\right)
  .
\end{align*}
Then, for any $h_j >0$ and $n_j \geq 0$, for $j=1,\dots,s$, the following inequality is satisfied 
\begin{align}\label{dexpo:diserr}
  |I_{\bsh} (f) - Q_{\bsh}^{\mathscr{D}_{\bsn}}(f)|
  &\leq  
  C'(s, \nu, \bsh)
\exp\left[- e_* \exp\left(\inf_{j}c_j h_j^{d_j} \left(\frac{n_j+1}{2}\right)^{d_j}\right) \right]
  ,
\end{align}
where
\begin{align*}
C'(s,\nu, \bsh) 
:=
2\,s 
\prod_{j=1}^s 
\e^{-e_j} \left( h_j + 2 \frac{\Gamma(1/d_j)}{e_j \, c_j^{1/d_j} \, d_j} \right) 
.
\end{align*}
\end{theorem}
\begin{proof}
 Using the definition of the norm and the assumption on $\nu$, we get
\begin{align}\label{dexpo:ineq5}
 |I_{\bsh} (f) - Q_{\bsh}^{\mathscr{D}_{\bsn}}(f)| 
 &= 
 \left| h_1 \cdots h_s \sum_{\bsk \in \Z^s \setminus\mathscr{D}_{\bsn}} f(k_1 h_1, \dots, k_s h_s)\right|
 \notag
 \\
 &\le
 h_1 \cdots h_s \sum_{\bsk \in \Z^s \setminus\mathscr{D}_{\bsn}} \left| f(k_1 h_1, \dots, k_s h_s) \frac {\nu(k_1h_1, \dots, k_sh_s)}{\nu(k_1h_1, \dots, k_sh_s)} \right|
 \notag
 \\
 &\leq 
 \|f \, \nu\|_{L^\infty(\R^s)}
 \, h_1 \cdots h_s  
 \sum_{\bsk \in \Z^s \setminus\mathscr{D}_{\bsn}} \exp\left[-\sum_{j=1}^s e_j \exp\left(c_j |k_j\,h_j|^{d_j}\right)\right]
 .
\end{align}
Now we evaluate the summation in the last inequality. We have
\begin{align}\label{dexpo:ineq4}
& \sum_{\bsk \in \Z^s \setminus\mathscr{D}_{\bsn}} \exp
\left(
-\sum_{j=1}^s 
e_j \exp \left(c_j |k_j\,h_j|^{d_j}\right)
\right)
\notag
\\
&\qquad\qquad\leq
2 \sum_{{i}=1}^s
 \left[
 \sum_{k_1 \in \Z} \cdots \sum_{k_i= \frac{n_i}{2}+1} ^{\infty} \cdots\sum_{k_s \in \Z}
\exp
\left(
-\sum_{j=1}^s e_j \exp\left(c_j |k_j\,h_j|^{d_j}\right)
\right)
\right]
.
\end{align}
Since $d_i \geq 1$ then using~\eqref{lem:sumdoubexpo} and $n_i/2 + 1 \geq (n_i+1)/2$, we have
\begin{align}\label{dexpo:ineq6}
&\sum_{k_i=\frac{n_i}{2}+1}^{\infty} 
\exp
\left(
-e_i \exp
\left(
c_i k_i^{d_i}\,h_i^{d_i}
 \right)
 \right)
 \leq
 \exp\left[
 - e_i \exp\left(c_i h_i^{d_i} \, \left(\frac{n_i+1}{2} \right)^{d_i}\right) 
 \right]
 \e^{-e_i} \left( 1 + \frac{\Gamma(1/d_i)}{e_i \, c_i^{1/d_i} h_i \, d_i} 
 \right) 
 .
\end{align}
Moreover, using the inequality~\eqref{sumdoubexpo}, we obtain
\begin{align}\label{dexpo:ineq7}
\sum_{k_i \in \Z} 
\exp\left(-e_i \exp\left(c_i h_i ^{d_i} k_i^{d_i} \right)\right) 
&=
 2 \sum_{k_i=0}^{\infty} \exp\left(-e_i \exp\left(c_i h_i ^{d_i} k_i^{d_i} \right)\right)
 - \e^{-\e_i}
 \notag 
 \\
 & \leq
 \e^{-e_i} \left( 1 + 2 \frac{\Gamma(1/d_i)}{e_i \, c_i^{1/d_i} h_i \, d_i} \right) 
 .
\end{align}

Applying \eqref{dexpo:ineq6} and \eqref{dexpo:ineq7} into \eqref{dexpo:ineq4}, we get
\begin{align}\label{dexpo:ineq8}
& \sum_{\bsk \in  \Z^s \setminus\mathscr{D}_{\bsn}} \exp\left[-\sum_{j=1}^s e_j \exp \left(c_j |k_j\,h_j|^{d_j}\right)\right]
\notag
\\
&\leq 
2 
\sum_{i=1}^s 
\exp\left[ - e_i \exp\left(c_i h_i^{d_i} \, \left(\frac{n_i+1}{2} \right)^{d_i}\right) \right]
\prod_{j=1}^s 
\e^{-e_j} \left( 1 + 2 \frac{\Gamma(1/d_j)}{e_j \, c_j^{1/d_j} h_j \, d_j} \right) 
\notag
\\
&\leq
2 \, s \, 
\exp\left[- e_* \exp\left(\inf_{i} c_i h_i^{d_i} \left(\frac{n_i+1}{2}\right)^{d_i}\right) \right]
\prod_{j=1}^s 
\e^{-e_j} \left( 1 + 2 \frac{\Gamma(1/d_j)}{e_j \, c_j^{1/d_j} h_j \, d_j} \right) 
.
\end{align}
Using the above inequality into \eqref{dexpo:ineq5} yields the claim.
\end{proof}

Let us use the same notations for $C_*$ as in~\eqref{C*} and $C_\sharp$ as in~\eqref{C}. Now we obtain the total error for functions which decay double exponential.

\begin{theorem} \label{TheoDEC}
Suppose $f$ belongs to $E_s(\nu,\omega)$ for some $\nu$ and $\omega$ with conditions \ref{F-expexp} and \ref{W-exp} such that for some fixed $a_j, b_j, c_j, e_j> 0$ and $d_j \geq 1$, we can take
\begin{align*}
 \nu(\bsx) 
 =
 \exp
 \left
 (\sum_{j=1}^s e_j \exp\left(c_j \, |x_j|^{d_j}\right)
 \right),
\end{align*}
and
\begin{align*}
\omega (\bsxi)
 =
 \exp\left(\sum_{j=1}^s a_j |\xi_j|^{b_j}\right)
 .
\end{align*}
We set $h_j = \left( a_j h \right)^{1/b_j}$ with
\begin{align*}
h= 
N^{-1/B(s)} 
\left( \frac{\ln ( e_*^{-B(s)} N)}{ C_\sharp \, B(s) }\right)^{D(s)/B(s)}
,
\end{align*}
and
\begin{align*} 
n_j+1 
 = 
 \max 
 \left\{
 \left\lfloor 
 \frac
 	{ C_* h^\frac{B(s)}{D(s)\, d_j} N^{\frac{1}{D(s)\,d_j}}}
 	{C_j    h^{1/b_j}}
 \right\rfloor
 ,1 \right\}
 .
\end{align*}

Then, for all $N$ sufficiently large, the following inequality is satisfied  
\begin{align*}
|I(f) - Q_{\bsh}^{\mathscr{D}_{\bsn}}(f)|
 \leq  
  C'(s) 
  \exp\left
  (-N^{1/B(s)} 
  \left(\ln (e_* ^{-B(s)} N ) \right)^{-D(s)/B(s)} 
	C_\sharp^{D(s)/B(s)}  B(s)^{D(s)/B(s)}
  \right)
 \, \|f\| 
 ,
\end{align*}
where
\begin{align*}
C'(s) 
&:=
4\,s 
\prod_{j=1}^s 
\e^{-e_j} \left( a_j^{1/b_j}  + 2 \frac{\Gamma(1/d_j)}{e_j \, c_j^{1/d_j} \, d_j} \right) 
+
4\,\e \sum_{\emptyset \ne \setu \subseteq \{1:s\}}   \prod_{j\in\setu} \frac{\Gamma(1/b_j)}{b_j}
,
\end{align*}
and
\begin{align*}
(n_1+1)\cdots (n_s+1) \leq N.
\end{align*}
\end{theorem}
\begin{remark}\label{remN}
 $N$ is chosen sufficiently large such that $h \leq 1/ \ln(2\, \e)$, which is the condition on $h$ in \RefThm{thm:discretization}.
\end{remark}
\begin{proof}
Using \RefThm{thm:discretization} the discretization error does not change and is given by 
\begin{align}\label{dexpo:diserror3}
 |I(f) - I_{\bsh} (f)|
 \leq  
 C(s,\omega)\, \exp\left(-1/h\right)\,\|\hat f\, \omega \|_{L^\infty(\R^s)}
 .
\end{align}

The truncation error $|I_{\bsh} (f) - Q_{\bsh}^{\mathscr{D}_{\bsn}}(f)|$ is estimated using \RefThm{thm:DECtruncation}.
We set $h_j = (a_j h)^{1/b_j}$, $j=1, \dots, s$, the discretization error can be estimated as
\begin{align}\label{dexpo:diserr2}
 |I_{\bsh} (f) -  Q_{\bsh}^{\mathscr{D}_{\bsn}}(f)| 
  &\leq 
    C'(s,\nu) \, 
    \exp\left[- e_* \exp\left(\inf_{j}c_j a_j^{\frac{d_j}{b_j}} h^{\frac{d_j}{b_j}} \left(\frac{n_j+1}{2}\right)^{d_j}\right) 
    \right]\, 
    \|f \, \nu\|_{L^\infty(\R^s)}
  ,
\end{align}
where 
\begin{align*}
C'(s,\nu) 
:=
2\,s 
\prod_{j=1}^s 
\e^{-e_j} \left( a_j^{1/b_j} + 2 \frac{\Gamma(1/d_j)}{e_j \, c_j^{1/d_j} \, d_j} \right) 
.
\end{align*}

For all $j=1, \dots, s$ define $n_j$ as 
\begin{align*}
 n_j+1 
 = 
 \max 
 \left\{
 \left\lfloor 
 \frac
 	{ C_* h^\frac{B(s)}{D(s)\, d_j} N^{\frac{1}{D(s)\,d_j}}}
 	{C_j    h^{1/b_j}}
 \right\rfloor, 1 \right\}
 .
\end{align*}

Note that $\lfloor x \rfloor \geq \frac{x}{2}$, for all $x \geq 1$, the truncation error is rewritten as
\begin{align} \label{dexpo:ineq9}
  |I_{\bsh} (f) - Q_{\bsh}^{\mathscr{D}_{\bsn}}(f)| 
  & 
  \leq 
  C'(s,\nu) \,
  \exp
  \left(
  -e_* \exp\left(C_\sharp \, h^{\frac{B(s)}{D(s)}}N^{\frac{1}{D(s)}}  \right)
  \right)
  \|f \, \nu\|_{L^\infty(\R^s)} 
  .
\end{align}

Similarly to \RefThm{TheoEC}, by balancing the orders of magnitude of two errors we obtain 
\begin{align}\label{eq:h}
h^{-1} 
= 
e_* \exp\left(C_\sharp \, h^{\frac{B(s)}{D(s)}}N^{\frac{1}{D(s)}}  \right) 
\end{align}
which leads to
\begin{align}\label{hW}
h = 
N^{-1/B(s)} 
\left(
 \frac{D(s) \W\left(C_\sharp \, \frac{B(s)}{D(s)} \, e_*^{-B(s)/D(s)}  N^{1/D(s)}\right)}{C_\sharp\,B(s) } \right)^{D(s)/B(s)}
,
\end{align}
where $\W(\cdot)$ is the Lambert-W function.

Here, we will approximately solve the equation \eqref{eq:h} which is equivalent to
\begin{align*}
h= 
N^{-1/B(s)} 
\left( \frac{- \ln (e_* h) }{ C_\sharp }\right)^{D(s)/B(s)}
.
\end{align*}
Substituting $h = N^{-1/B(s)}$ into the right hand side of the above equality implies
\begin{align}\label{hlog}
h= 
N^{-1/B(s)} 
\left( \frac{ -\ln ( e_*  N^{-1/B(s)})}{ C_\sharp}\right)^{D(s)/B(s)}
,
\end{align}
which asymptotically agrees with~\eqref{hW} since $\W(c\,x)$ is equivalent to $\ln(x)$ when $x$ goes to infinity for any positive constant $c$.
Substituting the chosen $h$ as in~\eqref{hlog} into \eqref{dexpo:diserror3} and \eqref{dexpo:ineq9}, the total error is estimated by the combination of these two inequalities.
\end{proof}

\section{Numerical results}

In this section we present some numerical examples to illustrate the results of \RefThm{TheoEC} and \RefThm{TheoDEC}. We consider three toy examples of integration that might be interesting since they are seen in many applications, namely: integration with respect to Gaussian distribution, integration with respect to exponential distribution and Fourier-like integration. 
For the last two integrations we apply so-called double exponential transformations, which were firstly developed by Mori and Takahasi, see \cite{Mor05, Tak74, Tan09}. Particularly, we use the change of variables
\begin{align*}
I(f) 
= 
\int_{\Omega} f(\bsx) \rd \bsx 
= 
\int_{\R^s} f(\phi_1(u_1), \dots, \phi_s(u_s)) \prod_{j=1}^s \phi_j'(u_j) \rd {\bsu}
,
\end{align*}
where  $\Omega \subseteq \R^s$ and $\phi_j$, for $j=1,\ldots,s$, are suitable double exponential transformations. Generally, these transformations not only give single or double exponential decay of the function, but also keep the smoothness of the integrand. We will discuss more about which transformations should be used for each particular example in the further part.

In order to calculate with very high precision, our test are implemented in C++ using the \textit{Boost.Multiprecision} library which allows us to define variables with arbitrary decimal digit precision.
\begin{example} We consider the integral
\begin{align}\label{example1}
I(f) 
= 
\int_{\R^s} \exp \left(-\sum_{j=1}^s x_j^2 \right)\rd {\bsx}
=
	\pi^{s/2}
.
\end{align}
The Fourier transform is given as 
\begin{align*}
\hat f (\bsxi)
= 
	 \left( \sqrt{\pi}/2 \right)^{s}
	\exp\left(-\sum_{j=1}^s \pi^2 \xi_j^2\right)	 
	 .
\end{align*}
Since the integrand and its Fourier transform already decay exponentially fast, for this integrand there is no need to use double exponential transform. In \RefFig{fig:1} the relative error $|I(f) - Q_{\bsh}^{\mathscr{D}_{\bsn}}(f)| / |I(f)|$ is shown with respect to the number of function evaluations (log-log) where the step sizes are chosen as in \RefThm{TheoEC} and $\lambda$ is chosen approximately equal to $1$,  see \RefRem{reLambda}. We see that the relative error decays exponentially like $\mathcal{O}(\exp(-c\,N^{1/s}))$, where $c\approx 1.60,\, 1.57,\, 1.52,\, 1.32$ for the case $s=1,2,4,8$ respectively.
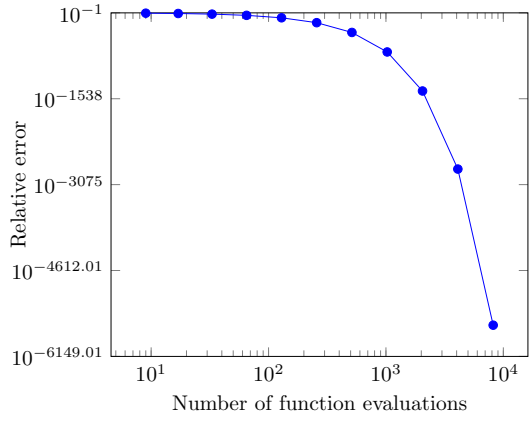
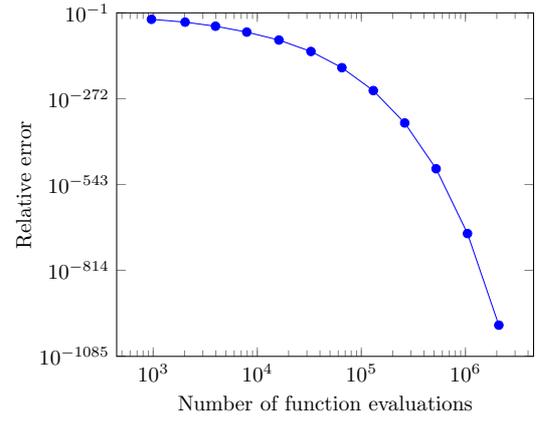
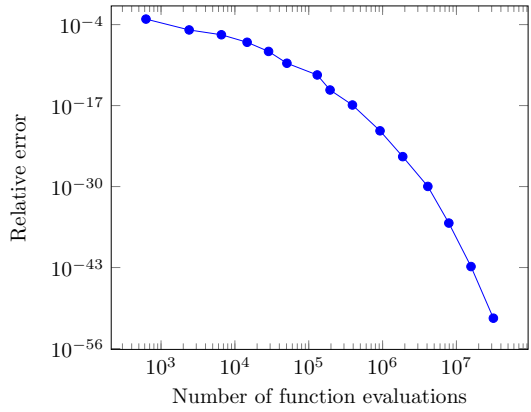
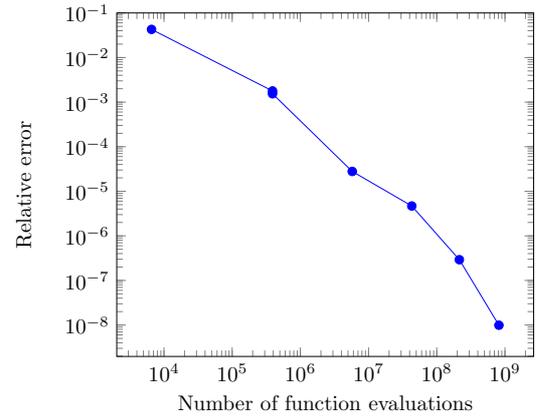
\begin{figure}
	\centering
	\begin{subfigure}[b]{0.4\textwidth}
		\centering
		\begin{tikzpicture}[scale=0.8]
		\begin{loglogaxis}
		[/pgf/number format/.cd,
        		1000 sep={},
        		xlabel={Number of function evaluations},
        		ylabel={Relative error},
        		ylabel style={yshift=0.3cm},
        		ymax = 0.1,  
        		legend pos = south east]	
		\addplot[color=blue, mark= *] table {trapz1D.txt};
		\end{loglogaxis}
		\end{tikzpicture}
	\caption{$s=1$}
	\end{subfigure}
	\hfill
	\begin{subfigure}[b]{0.4\textwidth}
		\centering
		\begin{tikzpicture}[scale=0.8]
		\begin{loglogaxis}
			[/pgf/number format/.cd,
        			1000 sep={},
        			xlabel={Number of function evaluations},
        			ylabel={Relative error},
        			ylabel style={yshift=0.3cm},
        			ymax = 0.1,  
        			legend pos = south east]
			\addplot[color=blue, mark= *] table {trapz2D.txt};
		\end{loglogaxis}
		\end{tikzpicture}
	\caption{$s=2$}
	\end{subfigure}
	\hfill
	\begin{subfigure}[b]{0.4\textwidth}
		\centering
		\begin{tikzpicture}[scale=0.8]
		\begin{loglogaxis}
			[/pgf/number format/.cd,
        			1000 sep={},
        			xlabel={Number of function evaluations},
        			ylabel={Relative error},
        			ylabel style={yshift=0.3cm},
        			ymax = 0.1,  
        			legend pos = south east]
			\addplot[color=blue, mark= *] table {trapz4D.txt};
		\end{loglogaxis}
		\end{tikzpicture}
	\caption{$s=4$}
	\end{subfigure}
	\hfill
	\begin{subfigure}[b]{0.4\textwidth}
		\centering
		\begin{tikzpicture}[scale=0.8]
		\begin{loglogaxis}
			[/pgf/number format/.cd,
        			1000 sep={},
        			xlabel={Number of function evaluations},
        			ylabel={Relative error},
        			ylabel style={yshift=0.3cm},
        			ymax = 0.1,  
        			legend pos = south east]
			\addplot[color=blue, mark= *] table {trapz8D_4.txt};			
		\end{loglogaxis}
		\end{tikzpicture}
	\caption{$s=8$}
	\end{subfigure}
	\caption{The relative error of the integration $\int_{\R^s} \exp(-\sum_{j=1}^s x_j^2) \rd \bsx$ with respect to the number of function evaluations.}
	\label{fig:1}
\end{figure}

\end{example}

\begin{example} To provide numerical examples that illustrate the convergence rate using double exponential transformation, we consider the integration over $\R_+^s$ with respect to the exponential probability density
\begin{align*}
I(f) 
= 
\int_{\R_{+} ^s} f (\bsx)\rd {\bsx}
:=
\int_{\R_{+} ^s} \prod_{j=1}^s x_j^2  \exp \left(\sum_{j=1}^s-x_j \right)\rd {\bsx}.
\end{align*}
Since $\exp \left(\sum_{j=1}^s-x_j \right)$ and $g(\bsx)=\prod_{j=1}^s x_j^2$ are entire functions,  due to the Paley--Wiener theorem discussed in \RefSec{Introduction} the Fourier transform $\hat f$ decays exponentially fast.

To get double exponential decay towards both $-\infty$ and $+\infty$ we apply the following double exponential transformations:

\begin{align*}
x_j = \phi_j(u_j):=\exp(u_j - \exp(-u_j))
,
\end{align*}
for $j =1,\dots, s$. This results in the integral
\begin{align*}
I(f) 
= 
\int_{\R^s} 	
	f(\exp(\bsu - \exp(-\bsu)))
	\prod_{j=1}^s \left( 1 + \exp(-u_j)\right) \exp\left(u_j -\exp\left(-u_j\right)\right)
	\rd {\bsu}
	.
\end{align*}
Applying our proposed method where the step sizes and truncation points are chosen as in \RefThm{TheoDEC}, the obtained result is displayed in \RefFig{fig:2} (log-log). We can see that the convergence rate again matches well with the expected convergence rate in the~\RefThm{TheoDEC}, i.e., of order $\mathcal{O}(\exp(-c\, \,N^{1/s} /\log(N))$, where $c\approx 4.37,\, 6.10,\,10.50,\, 9.63$  for the case $s=1,2,4,8$ respectively.

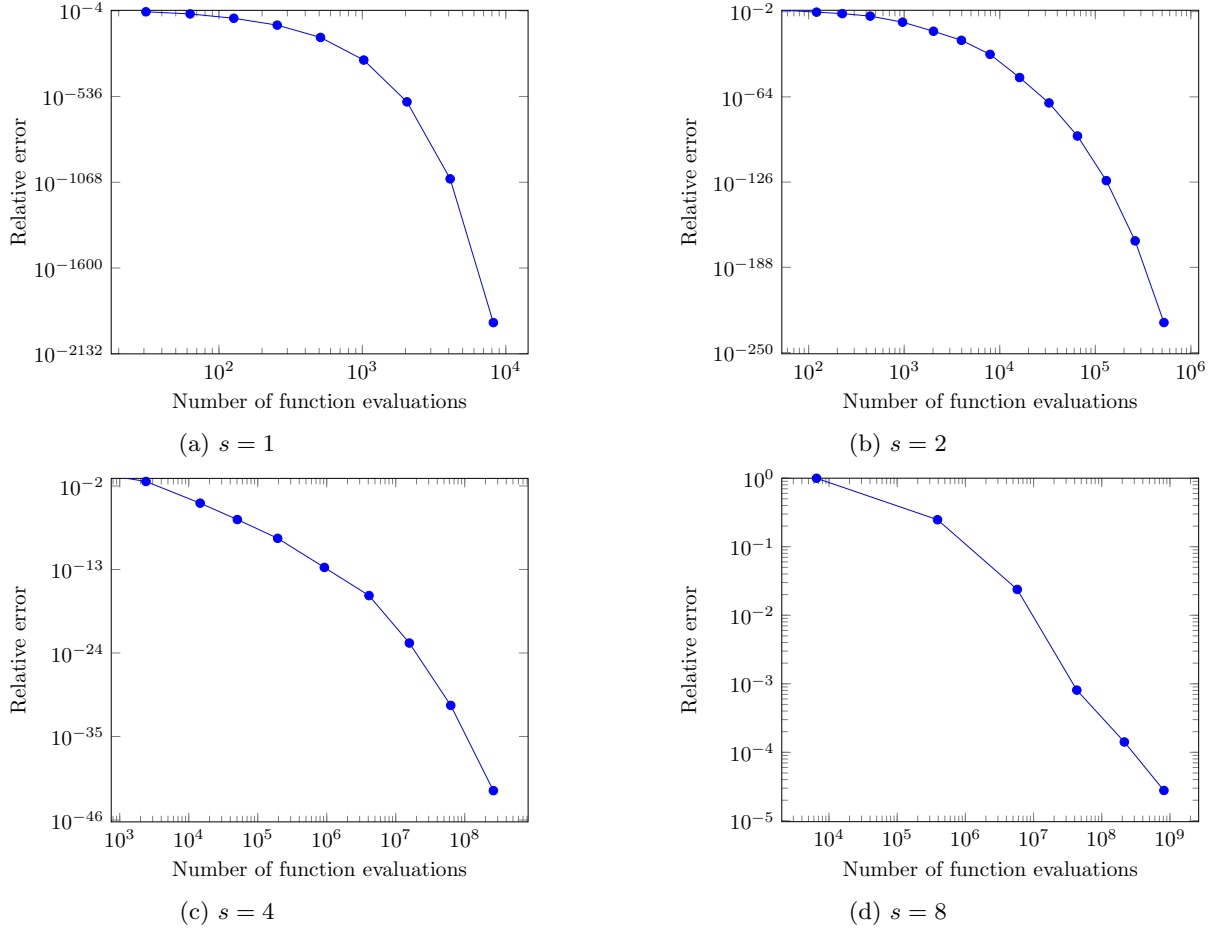
\begin{figure}
	\centering
	\begin{subfigure}[b]{0.4\textwidth}
		\centering
		\begin{tikzpicture}[scale=0.8]
		\begin{loglogaxis}
			[/pgf/number format/.cd,
        			1000 sep={},
        			xlabel={Number of function evaluations},
        			ylabel={Relative error},
        			ylabel style={yshift=0.3cm},
        			ymax = 0.1,  
        			legend pos = south east]
        		\addplot[color=blue, mark= *] table {double1D.txt};
		\end{loglogaxis}
		\end{tikzpicture}
	\caption{$s=1$}
	\end{subfigure}
	\hfill
	\begin{subfigure}[b]{0.4\textwidth}
		\centering
		\begin{tikzpicture}[scale=0.8]
		\begin{loglogaxis}
			[/pgf/number format/.cd,
        			1000 sep={},
        			xlabel={Number of function evaluations},
        			ylabel={Relative error},
        			ylabel style={yshift=0.3cm},
        			ymax = 0.1,  
        			legend pos = south east]
			\addplot[color=blue, mark= *] table {double2D.txt};
			
		\end{loglogaxis}
		\end{tikzpicture}
	\caption{$s=2$}
	\end{subfigure}
	\hfill
	\begin{subfigure}[b]{0.4\textwidth}
		\centering
		\begin{tikzpicture}[scale=0.8]
		\begin{loglogaxis}
			[/pgf/number format/.cd,
        			1000 sep={},
        			xlabel={Number of function evaluations},
        			ylabel={Relative error},
        			ylabel style={yshift=0.3cm},
        			ymax = 0.1,  
        			legend pos = south east]
			\addplot[color=blue, mark= *] table {double4D_lambda085.txt};
			
		\end{loglogaxis}
		\end{tikzpicture}
	\caption{$s=4$}
	\end{subfigure}	
	\hfill
	\begin{subfigure}[b]{0.4\textwidth}
		\centering
		\begin{tikzpicture}[scale=0.8]
		\begin{loglogaxis}
			[/pgf/number format/.cd,
        			1000 sep={},
        			xlabel={Number of function evaluations},
        			ylabel={Relative error},
        			ylabel style={yshift=0.3cm},
        			ymax = 1,  
        			legend pos = south east]
			\addplot[color=blue, mark= *] table {double8D_lambda08.txt};
		\end{loglogaxis}
		\end{tikzpicture}
	\caption{$s=8$}
	\end{subfigure}
	\caption{The relative error of the integration $\int_{\R_+^s} \prod_{j=1}^s x_j^2 \exp(-\sum_{j=1}^s x_j) \rd \bsx$ with respect to the number of function evaluations.}
	\label{fig:2}
\end{figure}
\end{example}
\begin{example}In this example we consider a Fourier-type integration 
\begin{align*}
I(f) = \int_{\R_{+}^s} \prod_{j=1}^s \frac{\sin(x_j)}{x_j} \rd \bsx
.
\end{align*}

We apply the following double exponential transformations, see \cite[\S 5.4]{Mor05} and \cite{Oou99}, for $j=1, \ldots,s$,
\begin{align}\label{FTDE1}
 x_j = M_j \, \phi_j(u_j):= \frac{M_j \, u_j}{1 -\exp (-2 \, u_j - \alpha_j (1-\e^{-u_j}) - \beta (\e^{u_j} -1))},
\end{align}
where $\beta = \frac{1}{4}$, $\alpha_j = \beta / \sqrt{1 + M_j \log (1+M_j)/ 4\pi}$ and $M_j$ is chosen depending on $h_j$ as
\begin{align}\label{hM}
M_j = \pi / h_j,
\end{align}
where $h_j$ is the step size. This results in the following integral
\begin{align}\label{intex3}
I(f) = \int_{\R^s} \prod_{j=1}^s \frac{\sin(M_j \, \phi_j(u_j))}{ \phi_j(u_j)} \phi_j'(u_j)\rd \bsu
:= I(g)
.
\end{align}
   
\begin{figure}
\centering
\begin{tikzpicture}
\begin{axis}
[ line width=0.5pt, xlabel = {$u_j$}, ylabel={$g(u_j)$}, axis lines=middle]
\addplot [smooth, color=blue, no markers] table {gu.txt};
\addplot [ only marks, mark size=1pt] table{trap_points.txt};
\end{axis}
\end{tikzpicture}
\caption{The graph of the function $g(u_j):=\frac{\sin(M_j \, \phi_j(u_j))}{ \phi_j(u_j)} \phi_j'(u_j)$ with $M_j =10$ and the function values at the trapezoidal points (black dots).}
       \label{ex3}
\end{figure}
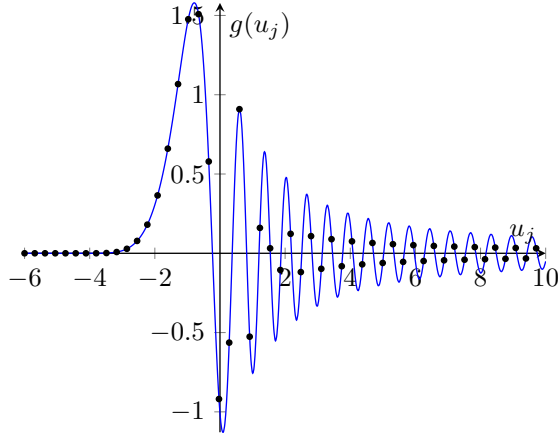

For large negative $u_j$, the univariate transformed integrand decays double exponentially. For large positive $u_j$, in contrary to the other double exponential transformations,~\eqref{FTDE1} does not improve the decay of the integrand. Instead, it guarantees the discretized transformed function (on the trapezoidal points) decays double exponentially fast, which is in fact the more accurate condition for the truncation error to hold. Indeed, when $n_j$ goes to $+\infty$, with the chosen $M_j$ as in~\eqref{hM} we have
\begin{align*}
\sin(M_j \, \phi_j(n_j h_j)) &= \sin\left( \frac{M_j \, n_j h_j}{1 -\exp (-2 \, n_j h_j - \alpha_j (1-\e^{-n_j h_j}) - \beta (\e^{n_j h_j} -1))}\right)
\\
&=\sin\left( \frac{\pi \, n_j }{1 -\exp (-2 \, n_j h_j - \alpha_j (1-\e^{-n_j h_j}) - \beta (\e^{n_j h_j} -1))}\right)
\\
&
\approx
(-1)^{n_j} \pi \, n_j \exp (-2 \, n_j h_j - \alpha_j (1-\e^{-n_j h_j}) - \beta (\e^{n_j h_j} -1))
,
\end{align*}
due to the fact that $\sin\left(\pi n/(1-x)\right) = \sin(\pi n + \pi  n x/(1-x))\approx (-1)^n \pi  n x$ when $x \approx 0$. Since the right hand side of the last equality decays double exponentially fast when $n$ goes to $+\infty$ and $\phi'(n_j h_j)/\phi(n_j h_j)$ is bounded, then the univariate transformed integrand becomes double exponentially small also at large positive trapezoidal points. As a result, we can truncate the trapezoidal sum at a moderate number of terms and expect an exponential convergence rate.
The graph of the univariate transformed function and the function values at the trapezoidal points are shown in \RefFig{ex3}.
The non-symmetry of the univariate transformed integrand leads to different behavior of the multivariate transformed integrand in different $s$ dimensional orthants. Therefore, in this example instead of truncating the integration domain in advance, we use an adaptive truncation strategy as was done in \cite{Oou99}.

Since the considered function is isotropic, $h_j$ and $M_j$ are chosen equal to $\hbar$ and $M$ respectively for all $j$. Due to the analyticity of the transformed integrand, its Fourier transform decays exponentially fast, presumably, with $a_j=a$ and $b_j=1$ for all $j$ and some positive constant $a$. Substituting chosen $h_j$, $a_j$ and $b_j$ into \RefThm{thm:discretization}, the discretization error can be estimated as
\begin{align*}
|I(g) - I_{\bsh}(g)| 
\leq
\, C(s,\omega) \,\|\hat g \, \omega\|_{L^\infty(\R^s)} \, \exp(-a/\hbar)
\approx
\exp(-a/\hbar)
.
\end{align*}
The value of $a$ will be determined numerically in order to achieve the best accuracy. The trapezoidal summation is truncated when its terms are small enough, namely when the following condition is satisfied
\begin{align*}
 \hbar^s \left| \prod_{j=1}^s \frac{\sin(M \, \phi_j(n_j \hbar))}{ \phi_j(n_j \hbar)} \phi_j'(n_j \hbar) \right| 
<  
\, \exp(-a/\hbar)
.
\end{align*} 

In the implementation, we repeat the test for $\hbar=\pi/M$ with an increasing sequence of $M$. We tested this problem numerically up to the 4 dimensional case. We choose $a=5$ for $s=1,2,3$ and $a=6$ for $s=4$. \RefFig{fig:3} (log-log) shows the relative error with respect to the number of function evaluations, which shows an expected rate of convergence close to $\mathcal{O}\left(\exp(-c\,N^{1/s} / \log N)\right)$, where $c\approx 2.17,\, 4.32,\, 5.11,\, 6.86$ for the case $s=1,2,3,4$ respectively.

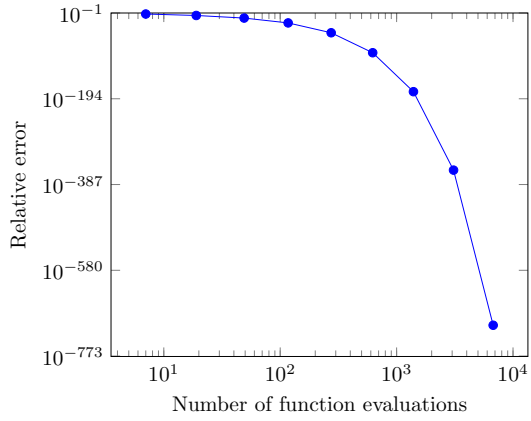
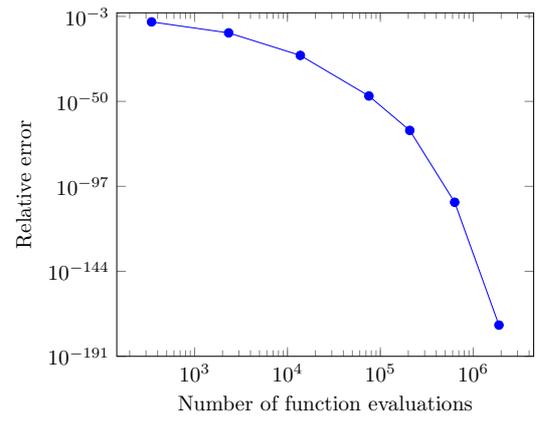
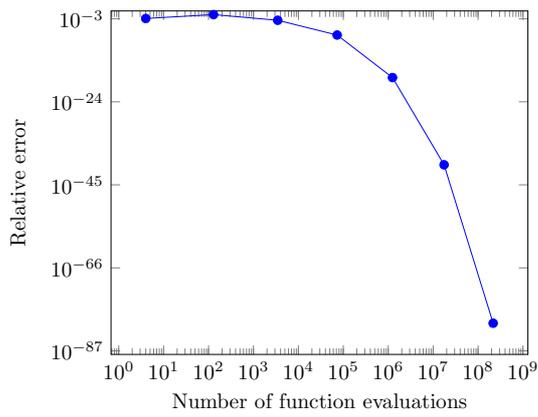
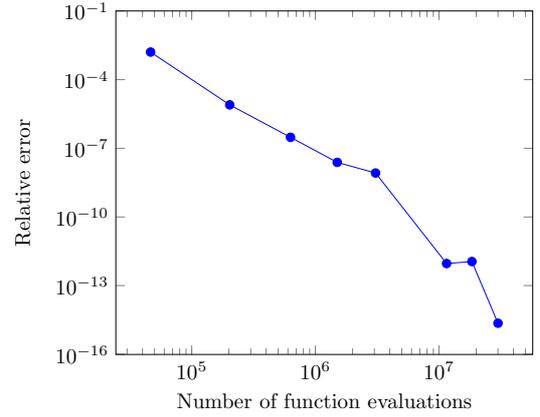
\begin{figure}
	\centering
	\begin{subfigure}[b]{0.4\textwidth}
		\centering
		\begin{tikzpicture}[scale=0.8]
		\begin{loglogaxis}
		[/pgf/number format/.cd,
        		1000 sep={},
        		xlabel={Number of function evaluations},
        		ylabel={Relative error},
        		ylabel style={yshift=0.3cm},
        		ymax = 0.1,  
        		legend pos = south east]
		\addplot[color=blue, mark= *] table {sinc1D.txt};
		\end{loglogaxis}
		\end{tikzpicture}
	\caption{$s=1$}
	\end{subfigure}
	\hfill
	\begin{subfigure}[b]{0.4\textwidth}
		\centering
		\begin{tikzpicture}[scale=0.8]
		\begin{loglogaxis}
			[/pgf/number format/.cd,
        			1000 sep={},
        			xlabel={Number of function evaluations},
        			ylabel={Relative error},
        			ylabel style={yshift=0.3cm},
        			ymax = 0.1,  
        			legend pos = south east]
			\addplot[color=blue, mark= *] table {sinc2D.txt};
		\end{loglogaxis}
		\end{tikzpicture}
	\caption{$s=2$}
	\end{subfigure}
	\hfill
	\begin{subfigure}[b]{0.4\textwidth}
		\centering
		\begin{tikzpicture}[scale=0.8]
		\begin{loglogaxis}
			[/pgf/number format/.cd,
        			1000 sep={},
        			xlabel={Number of function evaluations},
        			ylabel={Relative error},
        			ylabel style={yshift=0.3cm},
        			ymax = 0.1,  
        			legend pos = south east]
			\addplot[color=blue, mark= *] table {sinc3D.txt};
		\end{loglogaxis}
		\end{tikzpicture}
	\caption{$s=3$}
	\end{subfigure}
	\hfill
	\begin{subfigure}[b]{0.4\textwidth}
		\centering
		\begin{tikzpicture}[scale=0.8]
		\begin{loglogaxis}
			[/pgf/number format/.cd,
        			1000 sep={},
        			xlabel={Number of function evaluations},
        			ylabel={Relative error},
        			ylabel style={yshift=0.3cm},
        			ymax = 0.1,  
        			legend pos = south east]
			\addplot[color=blue, mark= *] table {sinc4D_C6.txt};
		\end{loglogaxis}
		\end{tikzpicture}
	\caption{$s=4$}
	\end{subfigure}
	\hfill
	\caption{The relative error of the integration $\int_{\R^s} \prod_{j=1}^s \sinc(x_j) \rd \bsx$ with respect to the number of function evaluations.}
	\label{fig:3}
\end{figure}

\end{example}
In view of the theoretical results and the results of numerical experiments, we may see the curse of dimensionality. This will be discussed more in \RefSec{Conclusion}.
\section{Conclusion}\label{Conclusion}
In this work we have analyzed the integration of analytic functions over Euclidean space, we show the truncated trapezoidal rule is efficient, i.e., the error decays exponentially fast, for a moderate number of dimensions. This setting allows us to estimate both the discretization error and truncation error; and then optimally balance them. 
We have also numerically verified our theory by some numerical tests. The obtained results match very well with the theoretical estimations. 

Finally, the convergence rates in \RefThm{TheoEC} and \RefThm{TheoDEC} depend strongly on the dimension $s$ via the parameters $B(s)$ and $D(s)$. One may hope to relax these dependences by making stronger assumptions on the smoothness and the decay of the functions, e.g., by assuming both $B(s)$ and $D(s)$ to be uniformly bounded when $s$ increases. This is equivalent to requiring  the function and its Fourier transform to be both concentrated at the origin which is unfortunately in conflict with a general phenomenon sometimes called ``the Heisenberg uncertainty problem'' \cite{Hir50}. Additionally, although our truncation strategy is the simplest one, easy to be analyzed and implemented, it has not inexhaustibly taken advantage of the trapezoidal rule. For example, for the integration of the Gaussian function as in~\eqref{example1} it is easy to see that it will be more efficient if we truncate the trapezoidal points within a sphere rather than a cube. Therefore, it is necessary to seek for other possible efficient truncation strategies. This will be left for future research.  
\section*{Acknowledgment}
We are grateful to two anonymous referees whose comments helped significantly improve and clarify this paper. The authors were supported by the KU Leuven research fund OT:3E130287.
\bibliography{IRs}

\begin{thebibliography}{10}

\bibitem{Dic11}
Josef Dick, Gerhard Larcher, Friedrich Pillichshammer, and Henryk
  Wo{\'z}niakowski.
\newblock Exponential convergence and tractability of multivariate integration
  for {K}orobov spaces.
\newblock {\em Mathematics of Computation}, 80(274):905--930, 2011.

\bibitem{Hir50}
II~Hirschman.
\newblock On the behaviour of {F}ourier transforms at infinity and on
  quasi-analytic classes of functions.
\newblock {\em American Journal of Mathematics}, 72(1):200--213, 1950.

\bibitem{Irr15}
Christian Irrgeher, Peter Kritzer, Gunther Leobacher, and Friedrich
  Pillichshammer.
\newblock Integration in {H}ermite spaces of analytic functions.
\newblock {\em Journal of Complexity}, 31(3):380--404, 2015.

\bibitem{Kri14}
Peter Kritzer, Friedrich Pillichshammer, and Henryk Wo{\'z}niakowski.
\newblock Multivariate integration of infinitely many times differentiable
  functions in weighted {K}orobov spaces.
\newblock {\em Mathematics of Computation}, 83(287):1189--1206, 2014.

\bibitem{Mor05}
Masatake Mori.
\newblock Discovery of the double exponential transformation and its
  developments.
\newblock {\em Publications of the Research Institute for Mathematical
  Sciences}, 41(4):897--935, 2005.

\bibitem{Oou99}
Takuya Ooura and Masatake Mori.
\newblock A robust double exponential formula for {F}ourier-type integrals.
\newblock {\em Journal of computational and applied mathematics},
  112(1-2):229--241, 1999.

\bibitem{Sin13}
Vasile Sinescu, Frances~Y Kuo, and Ian~H Sloan.
\newblock On the choice of weights in a function space for quasi-{M}onte
  {C}arlo methods for a class of generalised response models in statistics.
\newblock In {\em Monte Carlo and Quasi-Monte Carlo Methods 2012}, pages
  631--647. Springer, 2013.

\bibitem{Slo87}
Ian~H Sloan and TR~Osborn.
\newblock Multiple integration over bounded and unbounded regions.
\newblock {\em Journal of computational and applied mathematics},
  17(1-2):181--196, 1987.

\bibitem{Ste03}
Elias~M Stein and Rami Shakarchi.
\newblock Complex analysis. {P}rinceton lectures in analysis, {II}, 2003.

\bibitem{Sug87}
Masaaki Sugihara.
\newblock Method of good matrices for multi-dimensional numerical
  integrations--{A}n extension of the method of good lattice points.
\newblock {\em Journal of computational and applied mathematics},
  17(1-2):197--213, 1987.

\bibitem{Sug97}
Masaaki Sugihara.
\newblock Optimality of the double exponential formula--functional analysis
  approach--.
\newblock {\em Numerische Mathematik}, 75(3):379--395, 1997.

\bibitem{Tak74}
Hidetosi Takahasi and Masatake Mori.
\newblock Double exponential formulas for numerical integration.
\newblock {\em Publications of the Research Institute for Mathematical
  Sciences}, 9(3):721--741, 1974.

\bibitem{Tan09}
Kenichiro Tanaka, Masaaki Sugihara, Kazuo Murota, and Masatake Mori.
\newblock Function classes for double exponential integration formulas.
\newblock {\em Numerische Mathematik}, 111(4):631--655, 2009.

\bibitem{Tre14}
Lloyd~N Trefethen and JAC Weideman.
\newblock The exponentially convergent trapezoidal rule.
\newblock {\em SIAM Review}, 56(3):385--458, 2014.

\end{thebibliography}
\end{document}